\documentclass[12pt]{article}

\usepackage{amsmath, amssymb}
\usepackage{mathtools}
\usepackage{tikz}
\usepackage{geometry}
\usepackage{graphicx}
\usepackage{colonequals}
\usepackage{multirow}
\usetikzlibrary{3d, calc}
\usepackage{amsthm}
\usepackage{hyperref}

\newtheorem{theorem}{Theorem}[section]

\newtheorem{proposition}[theorem]{Proposition}
\newtheorem{corollary}{Corollary}[theorem]

\theoremstyle{definition}
\newtheorem{definition}[theorem]{Definition}
\newtheorem{example}[theorem]{Example}

\theoremstyle{remark}
\newtheorem{remark}[theorem]{Remark}

\newcommand{\tilesize}{0.8}      % tile size in cm
\newcommand{\knotthickness}{4pt} % strand thickness
\newcommand{\cspace}{0.16}       % crossing gap (fraction of \tilesize)
\newcommand{\tileedge}{0.4pt}    % tile border width

\tikzset{
  tilestyle/.style={draw=black, line width=\tileedge},
  knotstyle/.style={draw=black, line cap=round, line join=round, line width=\knotthickness},
  mosai/.style={
    execute at begin cell=\node\bgroup,
    execute at end cell=\egroup,
    column sep={\tilesize cm,between origins},
    row sep={\tilesize cm,between origins},
    inner sep=0, outer sep=0
  }
}

\newcommand{\tileo}{\tikz \draw[tilestyle] (0,0) rectangle (\tilesize,\tilesize);}

\newcommand{\tilei}[1][{}]{\tikz{
  \draw[tilestyle] (0,0) rectangle (\tilesize,\tilesize);
  \begin{scope}\clip (0,0) rectangle (\tilesize,\tilesize);
    \draw[#1,knotstyle] (0,.5*\tilesize) to[out=0,in=90] (.5*\tilesize,0);
  \end{scope}
}}

\newcommand{\tileii}[1][{}]{\tikz{
  \draw[tilestyle] (0,0) rectangle (\tilesize,\tilesize);
  \begin{scope}\clip (0,0) rectangle (\tilesize,\tilesize);
    \draw[#1,knotstyle] (\tilesize,.5*\tilesize) to[out=180,in=90] (.5*\tilesize,0);
  \end{scope}
}}

\newcommand{\tileiii}[1][{}]{\tikz{
  \draw[tilestyle] (0,0) rectangle (\tilesize,\tilesize);
  \begin{scope}\clip (0,0) rectangle (\tilesize,\tilesize);
    \draw[#1,knotstyle] (\tilesize,.5*\tilesize) to[out=180,in=-90] (.5*\tilesize,\tilesize);
  \end{scope}
}}

\newcommand{\tileiv}[1][{}]{\tikz{
  \draw[tilestyle] (0,0) rectangle (\tilesize,\tilesize);
  \begin{scope}\clip (0,0) rectangle (\tilesize,\tilesize);
    \draw[#1,knotstyle] (0,.5*\tilesize) to[out=0,in=-90] (.5*\tilesize,\tilesize);
  \end{scope}
}}

\newcommand{\tilev}[1][{}]{\tikz{
  \draw[tilestyle] (0,0) rectangle (\tilesize,\tilesize);
  \begin{scope}\clip (0,0) rectangle (\tilesize,\tilesize);
    \draw[#1,knotstyle] (0,.5*\tilesize) -- (\tilesize,.5*\tilesize);
  \end{scope}
}}

\newcommand{\tilevi}[1][{}]{\tikz{
  \draw[tilestyle] (0,0) rectangle (\tilesize,\tilesize);
  \begin{scope}\clip (0,0) rectangle (\tilesize,\tilesize);
    \draw[#1,knotstyle] (.5*\tilesize,0) -- (.5*\tilesize,\tilesize);
  \end{scope}
}}

\newcommand{\tilevii}[1][{}]{\tikz{
  \draw[tilestyle] (0,0) rectangle (\tilesize,\tilesize);
  \begin{scope}\clip (0,0) rectangle (\tilesize,\tilesize);
    \draw[#1,knotstyle] (0,.5*\tilesize) to[out=0,in=90] (.5*\tilesize,0);
    \draw[#1,knotstyle] (\tilesize,.5*\tilesize) to[out=180,in=-90] (.5*\tilesize,\tilesize);
  \end{scope}
}}

\newcommand{\tileviii}[1][{}]{\tikz{
  \draw[tilestyle] (0,0) rectangle (\tilesize,\tilesize);
  \begin{scope}\clip (0,0) rectangle (\tilesize,\tilesize);
    \draw[#1,knotstyle] (\tilesize,.5*\tilesize) to[out=180,in=90] (.5*\tilesize,0);
    \draw[#1,knotstyle] (0,.5*\tilesize) to[out=0,in=-90] (.5*\tilesize,\tilesize);
  \end{scope}
}}

\newcommand{\tileix}[1][{}]{\tikz{
  \draw[tilestyle] (0,0) rectangle (\tilesize,\tilesize);
  \begin{scope}\clip (0,0) rectangle (\tilesize,\tilesize);
    \draw[#1,knotstyle] (.5*\tilesize,0) -- (.5*\tilesize,.5*\tilesize-\tilesize*\cspace);
    \draw[#1,knotstyle] (.5*\tilesize,.5*\tilesize+\tilesize*\cspace) -- (.5*\tilesize,\tilesize);
    \draw[#1,knotstyle] (0,.5*\tilesize) -- (\tilesize,.5*\tilesize);
  \end{scope}
}}

\newcommand{\tilex}[1][{}]{\tikz{
  \draw[tilestyle] (0,0) rectangle (\tilesize,\tilesize);
  \begin{scope}\clip (0,0) rectangle (\tilesize,\tilesize);
    \draw[#1,knotstyle] (0,.5*\tilesize) -- (.5*\tilesize-\tilesize*\cspace,.5*\tilesize);
    \draw[#1,knotstyle] (.5*\tilesize+\tilesize*\cspace,.5*\tilesize) -- (\tilesize,.5*\tilesize);
    \draw[#1,knotstyle] (.5*\tilesize,0) -- (.5*\tilesize,\tilesize);
  \end{scope}
}}

\newcommand{\AllyThanks}[1]{%
   \thanks{%
      Department of Mathematics; %
      University of California, Davis; %
      One Shields Ave; %
      Davis, CA, 95616, USA; %
      \href{mailto:anagasawahinck@ucdavis.edu}{anagasawahinck@ucdavis.edu}; %
      \url{}. %
   #1%
   }%
}

\newcommand{\PeytonThanks}[1]{%
   \thanks{%
      Department of Mathematics; %
      University of California, Davis; %
      One Shields Ave; %
      Davis, CA, 95616, USA; %
      \href{mailto:ppwood@ucdavis.edu}{ppwood@ucdavis.edu}; %
      \url{https://sites.google.com/view/peytonpwood/}. %
   #1%
   }%
}

\usepackage{xcolor}
\usepackage[most]{tcolorbox}

\newtcolorbox{commentboxally}{
  colback=red!10,     % Background color
  colframe=black!80,    % Border color
  boxrule=0.8pt,         % Border thickness
  arc=4pt,               % Rounded corners
  left=6pt, right=6pt,   % Padding
  top=4pt, bottom=4pt,   % Padding
  fontupper=\small\itshape, % Text style inside box
}

\newtcolorbox{commentboxpeyton}{
  colback=teal!10,     % Background color
  colframe=black!80,    % Border color
  boxrule=0.8pt,         % Border thickness
  arc=4pt,               % Rounded corners
  left=6pt, right=6pt,   % Padding
  top=4pt, bottom=4pt,   % Padding
  fontupper=\small\itshape, % Text style inside box
}

\vspace{-4cm}
\title{Spherical Knot Mosaics}
\vspace{-0.5cm}
\author{%
   Ally Nagasawa-Hinck
   \AllyThanks{}%
   \and%
   Peyton Phinehas Wood%
   \PeytonThanks{}%
}
\date{\today}

\begin{document}

\maketitle

\vspace{-1cm}

\begin{abstract}
In this paper we introduce the notion of a spherical knot mosaic where a knot is represented by tiling the surface of a topological $S^2$ with 11 canonical knot mosaic tiles and show this gives rise to several novel knot (and link) invariants: the spherical mosaic number, spherical tiling number, minimal spherical mosaic tiling number, spherical face number, spherical $n$-mosaic face number, and minimal spherical mosaic face number. We show examples where this framework offers an improvement over classical knot mosaics. Furthermore, we explore several bounds involving classical knot invariants derived from these spherical mosaic invariants.
\end{abstract}

\vspace{-0.5cm}
\section{Introduction}

A \textit{tame topological knot}, which we will refer to as a {\em knot} or {\em classical knot} throughout this paper, is a topological embedding of $S^1$ into $S^3$ which is ambient isotopic to a polygonal or smooth embedding of $S^1$ into $S^3$. Knots can be studied using knot diagrams. Knot mosaics which we will sometimes refer to as {\em classical knot mosaics} impose further restrictions on knot diagrams by having one represent a knot diagram using eleven $1 \times 1$ mosaic tiles on an $n \times n$ grid. This additional structure allows one to define knot invariants such as the mosaic number. We discuss this background with greater detail in \autoref{preliminaries}.

 In \autoref{spherical} we extend the notion of knot mosaics by introducing \textit{spherical knot mosaics} where one considers a knot projected on $S^2 \subset \mathbb{R}^3$ which we represent as the boundary of a cube of side length $n$ where each face is an $n \times n$ grid tiled by $1 \times 1$ mosaic tiles and faces connect across adjacent edges. We also define the {\em planar representation} of a spherical knot mosaic making this an easier object to study. Examples of planar representations for each of the first 35 unoriented prime knots can be found in \autoref{app:planar}.

In \autoref{invariants} we introduce several novel knot invariants using spherical knot mosaics and show examples where spherical knot mosaic invariants are stronger than classical knot mosaic invariants. In \autoref{bounds} we show several bounds on our spherical knot mosaic invariants involving classical knot invariants. In particular, we prove the following theorems:

\begin{theorem}[\autoref{m(K))-2}]
    For all knots, the spherical mosaic number is less than or equal to the classical mosaic number minus two,

    \[\text{i.e. } sm(K) \le m(K)-2.\]
\end{theorem}

\begin{theorem}[\autoref{t(K)-1}]
    For all knots, the spherical tiling number is less than or equal to the classical mosaic tiling number minus one,

    \[\text{i.e. } st(K) \le t(K)-1.\]
\end{theorem}

\begin{theorem}[\autoref{crossingnumbound}]\label{crossnumbound1.3}
    If a spherical $n$-mosaic represents a knot, then it cannot have more than $6n^2-3n+1$ crossings.
\end{theorem}

After initially posting a preprint of this paper to ar$\chi$iv.org, the authors discovered a paper published by Pezzimenti, Meintel, and Shabon defining an object equivalent to a spherical knot mosaic called a {\em cubic knot mosaic} and containing an alternative proof of \autoref{crossnumbound1.3}, \cite[Theorem~3.2 p. 634]{CubicMosaics}, and a result equivalent to combining the result of \autoref{cor:ge2} with \autoref{app:planar}, \cite[Theorem~2.5 p. 634]{CubicMosaics}. We have updated this version of our paper to note wherever we have developed the same or a similar notion as \cite{CubicMosaics}.

%---------------------------------------------------------------------------------------------------------
%---------------------------------------------------------------------------------------------------------

\section{Background}\label{preliminaries}

Formally, a {\em (topological) knot type}, $\mathcal{K}$, is an equivalence class of (topological) embeddings of $S^1$ into $S^3$ under the relation of (topological) ambient isotopy. A specific, fixed embedding of $\mathcal{K}$ is a {\em knot}, $K$, i.e. \[\mathcal{K} \colonequals [K] = \{K' \in \mathcal{K} | K'\sim K\} \] An {\em ambient isotopy} between two knots is a continuous deformation through homeomorphisms of $S^3$, i.e. if $K_0, K_1: S^1 \hookrightarrow S^3$ are two knots then an ambient isotopy is a function $H_t : S^3 \to S^3$ such that $H_0(K_0) = K_0$, $H_1(K_0) = K_1$, and $H_t$ is a homeomorphism for all $t \in [0,1]$. There are similar notions for {\em smooth knot type} and {\em piecewise-linear (PL) knot type}. A (topological) knot is called {\em tame} if it is ambient isotopic to a PL or smooth knot. Otherwise, a (topological) knot is called {\em wild}. It was shown in \cite{Moise1952AffineSI} that tame topological, smooth, and PL knots are equivalent as categories and thus result in the same knot type classification.

Given a fixed embedding $K \in \mathcal{K}$ we can study the equivalence class of $K$ by considering a two dimensional regular projection called a \textit{knot diagram} where the only singularities allowed are transverse double points and each singularity is written in such a way to denote an over and under arc structure as seen in \autoref{fig:crossinginfo}. We call a knot diagram with the same knot type as $K$ a {\em representative} of $K$ or say that the diagram {\em represents} $K$. These diagrams give rise to several knot invariants such as the \textit{crossing number} of a knot, denoted $c(K)$, which is the minimum number of crossings among all possible knot diagrams representing $K$.

A {\em link} or {\em classical link} is the disjoint union of one or more knots. All knots are also one component links. Links also give rise to link diagrams. Each component of a link may either be linked with at least one other component such that all possible representations of that component overlap with one or more other components or the diagram can be split into separate non-overlapping components. For brevity's sake we will only refer to knots throughout the rest of this paper, but many of the notions discussed can also be defined for links.

\begin{figure}[ht]
    \centering
    \includegraphics[width=0.45\linewidth]{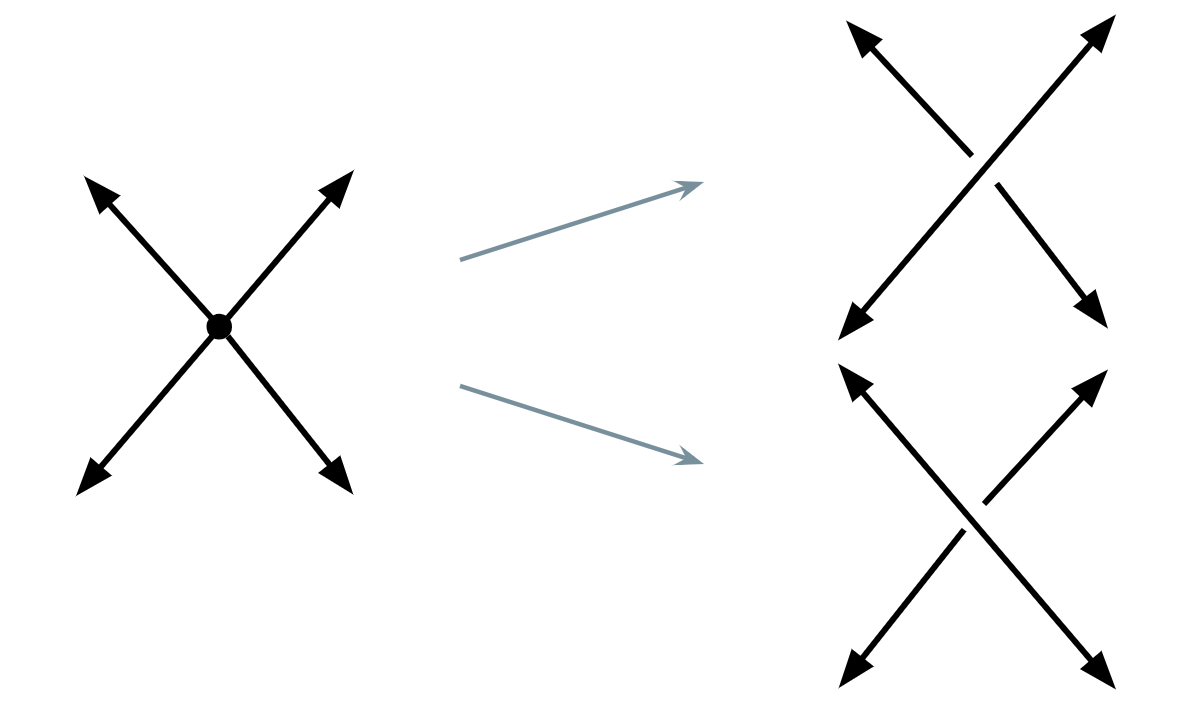}
    \caption{Showing the two ways one can add crossing information to a transverse double point of an unoriented regular projection. The arrows indicate each arc extends further into the arbitrary knot diagram and are not meant to convey an orientation.}
    \label{fig:crossinginfo}
\end{figure}

A \textit{knot $n$-mosaic} or {\em knot mosaic of size n} is an $n \times n$ grid where each entry is one of 11 \textit{mosaic tiles} depicted in \autoref{fig:tileset} arranged so that endpoints of arcs on adjacent tiles match up in such a way to result in a knot diagram. We call the set $\mathcal{T} = \{T_0, \dots, T_{10}\}$ the {\em tile set}, $T_0$ the {\em empty tile} and the other ten tiles {\em non-empty tiles}. The midpoint of an edge of a mosaic tile is called a \textit{connection point} if it is also the endpoint of an arc drawn on that tile. Each tile can have zero, two or four connection points.

One could also define an $n$-mosaic to be an $n \times n$ matrix such that each entry is one of the eleven mosaic tiles, \[\text{i.e. } M=(M_{ij})_{1 \le i,j \le n}, \hspace{.2in} M_{ij} \in \mathcal{T}.\] When the arcs on an $n$-mosaic form a single closed curve, then it is a knot $n$-mosaic \cite{Mosaic-number-of-knots}. A knot mosaic of a specific knot, $K$, can be denoted $M_K$.

\renewcommand{\knotthickness}{2.5pt}

\begin{figure}[ht]
\centering
\begin{tabular}{*{11}{c}}
\tileo & \tilei & \tileii & \tileiii & \tileiv &
\tilev & \tilevi & \tilevii & \tileviii & \tileix & \tilex \\
 $T_0$ & $T_1$ & $T_2$ & $T_3$ & $T_4$ &
$T_5$ & $T_6$ & $T_7$ & $T_8$ & $T_9$ & $T_{10}$
\end{tabular}
\caption{The 11 mosaic tiles which form the tile set}
\label{fig:tileset}
\end{figure}

In 2008, \cite{LomonacoKauffman2008} first defined knot mosaics as a discrete tool for representing quantum knots and conjectured that it would also be a useful framework for classical knots. Later that year in \cite{LOMONACO–KAUFFMAN-CONJECTURE}, it was shown that mosaic knot theory and classical knot theory are equivalent. That is, it was shown that knot mosaic type is a complete knot invariant, so two knot mosaics represent the same knot if and only if the knot diagrams belong to the same knot type. Since then, knot mosaics have been widely used to study classical knots.

\begin{figure}[!ht]
\begin{center}
\includegraphics[scale=0.65]{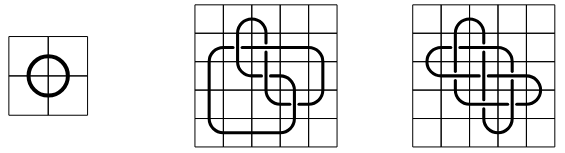}
\end{center}

\caption{Knot mosaics of the unknot (left), $4_1$ (center), and $7_4$ (right) \cite{Lee_2018}}
\label{fig:unknotmosaic}
\end{figure}

The most obvious invariant one can define with the added structure of an $n$-mosaic is the {\em mosaic number}, $m(K)$, of a knot. This is the smallest integer $n$ for which a knot can be represented as a knot $n$-mosaic. For example, the unknot, $0_1$, has mosaic number equal to 2, meaning it cannot be represented on a 1-mosaic, but can on a 2-mosaic as in \autoref{fig:unknotmosaic}. 

One can also consider the knot {\em tiling number}, $t(K)$, which is the minimal number of non-empty tiles needed to represent $K$ as a knot mosaic. If $M_K$ achieves its tiling number, it is called {\em space efficient} \cite{Heap_2018}. Closely related to these invariants is the \textit{minimal mosaic tile number}, $t_M(K)$, which is the minimum number of non-empty tiles required to construct a knot mosaic across all knot mosaic of size $m(K)$ which represent $K$. If $M_K$ both achieves its tiling number and has size $m(K)$, then it is called {\em minimally space efficient} \cite{Heap_2018}. See \cite[Theorem~8 on p. 788]{HeapKnowles2019} for proof that the tiling number of a knot can be different than the minimal mosaic tile number of a knot. We restate their example of $9_{10}$ in \autoref{tab:mosaicexamples}. The tiling number serves as a lower bound for the minimal mosaic tile number.  Since it is possible for $t(K)  < t_M(K)$ this means some knots only achieve their tiling number on a mosaic of size strictly greater than $m(K)$.

    \begin{table}[htbp]
    \centering
    
    \begin{tabular}{|c||c|c|c|c|c|}
        \hline
        & \multicolumn{5}{|c|}{Knot} \\
        
        \hline
        Invariant & $0_1$ & $3_1$ & $4_1$ & $5_1$ & $9_{10}$ \\
        \hline
        \hline
        $m(K)$ & 2 & 4 & $5^{\text{\cite{Lee_2018}}}$ & $5^{\text{\cite{Lee_2018}}}$ & $6^{\text{\cite[Thm 6]{HeapKnowles2019}}}$\\
        $t(K)$ & 4 & $12^{\text{\cite[Cor 12]{Heap_2018}}}$ & $17^{\text{\cite[Thm 13]{Heap_2018}}}$ & $17^{\text{\cite[Thm 13]{Heap_2018}}}$ &
        $27^{\text{\cite[Thm 8]{HeapKnowles2019}}}$\\
        $t_M(K)$ & 4 & $12^{\text{\cite[Cor 12]{Heap_2018}}}$ & $17^{\text{\cite[Thm 13]{Heap_2018}}}$ & $17^{\text{\cite[Thm 13]{Heap_2018}}}$ & $32^{\text{\cite[Thm 8]{HeapKnowles2019}}}$\\
        \hline
    \end{tabular}
    \caption{The mosaic number, tiling number, and minimal mosaic tile number for $0_1$, $3_1$, $4_1$, $5_1$, and $9_{10}$ with appropriate citations}
    \label{tab:mosaicexamples}
\end{table}

\vspace{-.2in}
\section{Spherical Knot Mosaics}\label{spherical}
We extend the notion of a classical mosaic on an $n \times n$ grid by tiling the faces of an $n \times n \times n$ cube embedded as $[0,n]\times[0,n]\times[0,n] \subset \mathbb{R}^3$. Under this notion each face of the cube is an $n \times n$ grid we tile using the same tile set as classical knot mosaics. Since topologically, a cube is homeomorphic to $S^2$, we call such a tiling a \textit{spherical $n$-mosaic} or {\em spherical mosaic of size n}. When the tiling results in a single closed component, we call it a \textit{spherical knot $n$-mosaic}. A spherical knot mosaic representing $K$ can be denoted $SM_K$.

\begin{figure}[ht]
    \centering
    \includegraphics[scale=.8]{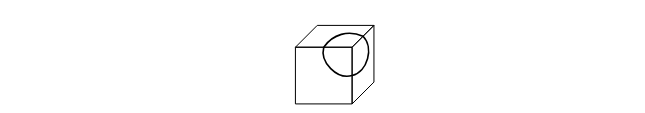}
    \caption{An example of the unknot as a spherical 1-mosaic}
    \label{fig:3DCube}
\end{figure}

Since it is both onerous and unenlightening to display more complicated 3D spherical knot mosaics in a 2D medium, we will break down each cube into a \textit{planar representation of $S^2$}. This visualization will be formed by six $n \times n$ grids arranged in a $T$ shape with the appropriate edges identified.
\vspace{-.2in}
\begin{figure}[!ht]
    \centering
    \vspace{.5cm}
    \includegraphics[width=0.2\linewidth]{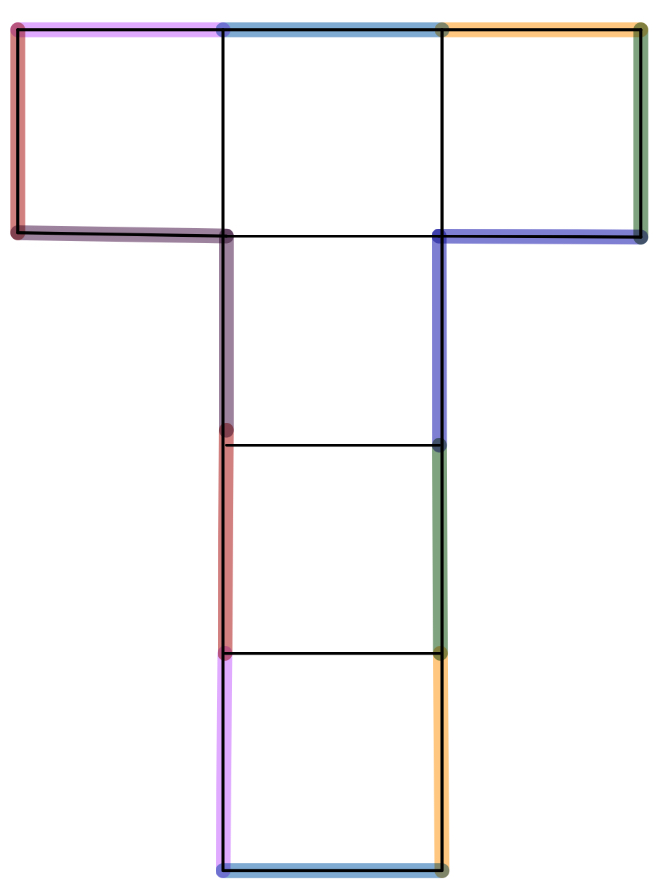}
    \caption{Planar representation of the cube with identified edges color-coded}
    \label{fig:cm41}
\end{figure}

We now have a variety of options for representing the unknot as a spherical knot 1-mosaic, some of which are seen in \autoref{fig:unknotplanar}. Examples of planar representations of the first 35 non-trivial knots with crossing number less than or equal to 8 can be found in \autoref{app:planar}.

\begin{figure}[!ht]
    \begin{center}
    \includegraphics[scale=0.8]{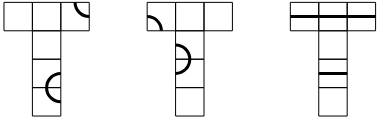}
    \end{center}
    \vspace{-.2in}
    \caption{Three examples of a planar representation of the unknot}
    \label{fig:unknotplanar}
\end{figure}

\begin{remark}
    It is interesting to point out that it is necessary we are referring to topological knots instead of PL or smooth knots throughout the rest of this paper as both PL and smooth can run into obstructions when trying to represent them as spherical knot mosaics. Since, as stated previously, topological, smooth, and PL knot theory are equivalent, we know these obstructions can all be dealt with, but as presented herein spherical knot mosaics will rarely represent a PL or smooth knot without some modification required.

    In particular, a spherical knot mosaic using any of $T_1, T_2, T_3, T_4, T_7, \text{ or } T_8$ from the tile set in \autoref{fig:tileset} does not represent a PL knot since there will be at least one tile that has a curved arc instead of a piecewise linear arc. However, the tile set can be easily modified so that these 6 tiles and thus any knot mosaic made with the modified tile set are PL.

    Classical knot mosaics represent both topological and smooth knots without any modification required, but the added structure of spherical knot mosaics disallow for smooth knots in most cases. Specifically, if $SM_K$ has non-empty tiles on 2 or more faces, then an arc will be connected across at least one edge of the cube. At these edges we see continuous, but non-differentiable points in our knot diagrams. One could resolve this issue by rounding the edges of the cube to recover smoothness.\\
\end{remark}

\section{Spherical Knot Mosaic Invariants}\label{invariants}

We now introduce several novel knot invariants arising from the structure of spherical knot mosaics. Our first novel knot invariant is an extension of the mosaic number. We define the {\em spherical mosaic number} of a knot, $K$, denoted $sm(K)$, to be the smallest integer, $n$, for which $K$ can be represented as a spherical knot $n$-mosaic. This is an equivalent notation to the cubic mosaic number in \cite[Definition 2.2 p. 632]{CubicMosaics}. We see in \autoref{app:planar} that the spherical mosaic number for the trefoil and figure-eight knot are 1 and that $sm(K) \le 2$ for all knots with $c(K) \le 8$. We establish that the unknot, trefoil, and figure-eight knots are the only knots with spherical mosaic number 1 in \autoref{cor:ge2} and thus all knots with crossing number 5 through 8 have $sm(K)=2$. These results are also shown in \cite[Proposition 2.3  p. 633]{CubicMosaics} and \cite[Theorem 2.5 p. 634]{CubicMosaics}.

Thus the spherical mosaic number distinguishes $4_1$ and $5_1$ since $sm(4_1)=1$ and $sm(5_1)=2$. We see in \autoref{tab:mosaicexamples} that the mosaic number, tiling number, and minimal mosaic tile number all fail to distinguish these two knots.

Our second novel invariant coming from spherical knot mosaics is an extension of the mosaic tiling number. Define the {\em spherical tiling number} of a knot, $K$, denoted $st(K)$, to be the smallest integer, $t$, for which $K$ can be represented as a spherical knot $n$-mosaic with $t$ non-empty mosaic tiles. One should note that this invariant does not require that the size of the mosaic $n$ be equal to $sm(K)$ when $st(K)$ tiles are used. Requiring such a restriction yields a third novel invariant based on the minimal mosaic tile number. We define the {\em minimal spherical mosaic tile number} of a knot, $K$, denoted $st_M(K)$ to be the smallest integer, $t_M$, for which $K$ can be represented as a spherical knot $sm(K)$-mosaic with $t_M$ non-empty mosaic tiles.

\begin{proposition}
    For a knot, $K$, $st(K) \le st_M(K)$.
\end{proposition}
    \begin{proof}
        This follows directly from the definition.
    \end{proof}

 As stated earlier, it is known that there are knots which have different tiling numbers and minimal mosaic tiling numbers \cite{HeapKnowles2019}. It is quite possible this is also the case for spherical knot mosaics, however finding such an example may prove to be harder than it was for classical mosaics given how efficiently spherical mosaics use space compared to classical mosaics.

Three additional novel invariants which do not arise from classical knot mosaic invariants are the minimal number of faces of a cube required to represent a knot, the minimal number of faces of a cube to represent a knot for a fixed $n$ and the minimal number of faces required to represent a knot when $n=sm(K)$. Define the {\em spherical face number} of a knot, $K$, denoted $sf(K)$ to be the smallest integer, $f$, for which $K$ can be represented as a spherical knot $n$-mosaic with $f$ non-empty faces. This is a highly ineffective invariant given that it will be equal to 1 for all knots since every knot can be represented on an $n$-mosaic for $n \ge m(K)$ and thus fit on one face of a spherical $n$-mosaic for $n \ge m(K)$ as well.

Define the {\em spherical $n$-mosaic face number} of a knot, $K$, and fixed $n$, denoted $sf_n(K)$, to be the smallest integer, $f_n$, where $K$ can be represented as a spherical knot $n$-mosaic with $f_n$ non-empty faces. Note this will be undefined when $n < sm(K)$. This definition also implies if $sm(K)=1$, meaning if $K$ is the unknot, trefoil, or figure eight knot, then $st(K)=sf_1(K)$.

\vspace{.4in}
\begin{proposition}\label{prop:whensf_ntrivial}
    For a knot $K$, if $n \ge m(K)$, then $sf_n(K) = 1$.
\end{proposition}
    \begin{proof}
        As each face of our spherical $n$-mosaic is at least an $m(K)$-mosaic by assumption, there exists a planar mosaic representation of $K$ which can be glued to any $m(K) \times m(K)$ sub-mosaic of any face of our spherical $n$-mosaic. Thus, $sf_n(K) = 1$.
    \end{proof}

 \begin{remark}
    Proposition \ref{prop:whensf_ntrivial} means $sf_n(K)$ will only be a stronger invariant than $sf(K)$, and thus, nontrivial, when fixing an $n$ such that $m(K) > n \ge sm(K)$. We show that one can always choose such an $n$, that is, such an $n$ will exist that yields a valid spherical knot $n$-mosaic in \autoref{m(K))-1} for all knots and get a tighter bound for nontrivial knots in \autoref{m(K))-2}.
 \end{remark}

 Define the {\em minimal spherical mosaic face number} of a knot $K$, denoted $sf_M(K)$, to be the smallest integer, $f_M$, for which $K$ can be represented as a spherical knot $n$-mosaic with $f_M$ non-empty faces where $n=sm(K)$.

\begin{example} Let's see some of the values these invariants take for the first five knots.
    \begin{table}[htbp]
    \centering
    
    \begin{tabular}{|c||c|c|c|c|c|}
        \hline
         & \multicolumn{5}{|c|}{Knot} \\
        \hline
        Invariant & $0_1$ & $3_1$ & $4_1$ & $5_1$ & $5_2$\\
        \hline
        \hline
        $sm(K)$   & 1 & 1 & 1 & 2 & 2\\
        $st(K)$   & 3 & 6 & 6 & $\le 14$ & $\le 14$\\
        $st_M(K)$ & 3 & 6 & 6 & $\le 14$ & $\le 14$\\
        $sf(K)$   & 1 & 1 & 1 & 1 & 1\\
        $sf_1(K)$ & 3 & 6 & 6 & - & -\\
        $sf_2(K)$ & 1 & 3 &$\le5$ & $\le 5$ & $\le 5$\\
        $sf_3(K)$ & 1 & 3 & 3 & 3 & 3\\
        $sf_M(K)$ & 3 & 6 & 6 & $\le 5$ & $\le 5$\\
        \hline
    \end{tabular}
    \caption{Some spherical mosaic theoretic invariants for $0_1$, $3_1$, $4_1$, $5_1$, and $5_2$. We use - to denote undefined. We give upper bounds where the exact value is not yet known.}
    \label{tab:unknottable}
\end{table}
\end{example}

\begin{proposition}
    For any knot, $K$, $sf_M(K) = sf_{sm(K)}(K)$, and $ 1 \le sf_{n}(K)\le sf_M(K)$ for all $n \ge sm(K)$.
\end{proposition}
\begin{proof}
    This follows directly from the definitions.
\end{proof}

This means that one can always represent a knot on a spherical $n$-mosaic for any $n$ larger than the minimum necessary and when doing so the minimal number of faces required may and eventually will decrease to 1 for large enough $n$. One could also consider a tuple made up of any number of these invariants. This will also be a knot invariant which may be stronger than any of these novel invariants alone.

%-------------------------------------------------------------------------
%-------------------------------------------------------------------------

\section{Bounds on Spherical Knot Mosaic Invariants}\label{bounds}

In the above section we noted several bounds that naturally arise between our novel invariants. In this section we will explore bounds that arise between our invariants and preexisting knot invariants such as mosaic number, tiling number, and crossing number.

\begin{proposition}\label{sm(K)<=m(K)}
    Given a knot, $K$, $sm(K) \le m(K)$.
\end{proposition}
\begin{proof}
    Consider a knot, $K$, and its corresponding minimal $m(K) \times m(K)$ mosaic. Copy this mosaic onto any face of a spherical $m(K)$-mosaic as in \autoref{fig:sm(K)<=m(K)}. Thus, $m(K)$ is an upper bound for $sm(K)$. 
\end{proof}
\vspace{-.4in}
\begin{figure}[ht]
\centering
\begin{minipage}[t]{0.18\textwidth}
\centering
\includegraphics[width=0.5\linewidth]{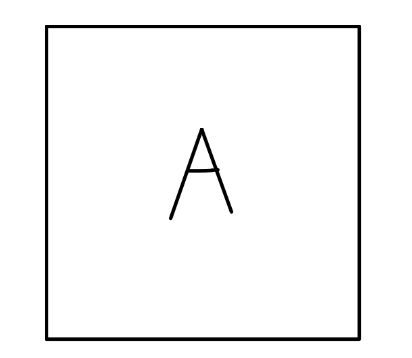}
\par\smallskip
{\footnotesize\textit{Planar $n$-mosaic $A$ representing some knot $K$}}
\end{minipage}
\hspace{0.4cm}
\begin{minipage}[t]{0.18\textwidth}
\centering
\includegraphics[width=0.5\linewidth]{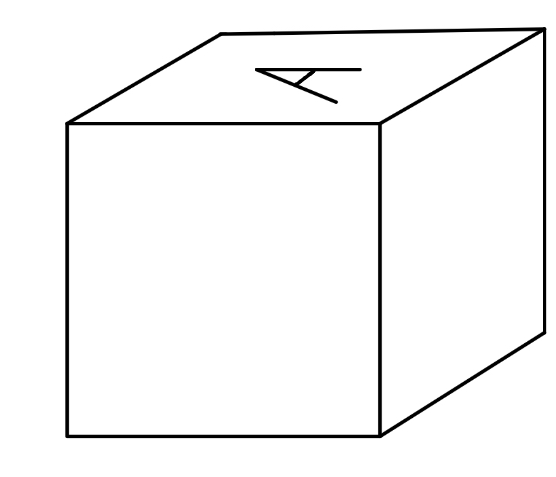}
\par\smallskip
{\footnotesize\textit{Spherical $n$-mosaic which also represents knot $K$}}
\end{minipage}
\caption{Visual depiction of Proposition \ref{sm(K)<=m(K)}}
\label{fig:sm(K)<=m(K)}
\end{figure}

This first bound is obvious from the definition and can be improved upon. We next show that this bound is not sharp.

\begin{theorem}\label{m(K))-1}
    Given a knot, $K$, $sm(K) \le m(K)-1$.
\end{theorem}

\begin{proof}
    Consider a knot, $K$, and any corresponding  $m(K) \times m(K)$ mosaic. Now consider any of the 4 possible $m(K)-1 \times m(K)-1$ sub-mosaics. This will contain all but $2m(K)-1$ boundary tiles of the original mosaic. Give this sub-mosaic onto any face of a spherical $(m(K)-1)$-mosaic and consider the remaining $2m(K)-1$ tiles which are now hanging off two adjacent sides of our spherical $(m(K)-1)$-mosaic.
    These two adjacent sides will have a shared corner tile. Delete it from the mosaic and glue the remaining $2m(K)-2$ non-corner boundary tiles to our spherical $(m(K)-1)$-mosaic tiles on the appropriate adjacent cube faces.
    
    The deleted tile will either be $T_0$ or one of $T_1$ through $T_4$. In either case, this operation of deleting the corner and mapping the remaining tiles to appropriate adjacent tiles on our cube preserves the knot type and closed-ness of our knot diagram. There is a depiction of this in figure 10.  Therefore, we have successfully created a spherical knot $(m(K)-1)$-mosaic representing $K$ showing that $sm(K) \le m(K) - 1$.
\end{proof}

\begin{figure}[ht]
\centering
\begin{tabular}{*{11}{c}}
\tileo & \tilei & \tileii & \tileiii & \tileiv &
\tilev & \tilevi & \tilevii & \tileviii & \tileix & \tilex \\
 $T_0$ & $T_1$ & $T_2$ & $T_3$ & $T_4$ &
$T_5$ & $T_6$ & $T_7$ & $T_8$ & $T_9$ & $T_{10}$
\end{tabular}
\caption{We repeat \autoref{fig:tileset} for the reader to reference.}
\label{fig:tilesetrepeat}
\end{figure}

This bound is sharp for the unknot, but otherwise can be improved further. We next show how the main idea in the proof of \autoref{m(K))-1} can be extended to get an even better bound for all non-trivial knots.

\begin{theorem}\label{m(K))-2}
    Given a nontrivial knot, $K$, $sm(K) \le m(K)-2$.
\end{theorem}

\begin{proof}
    Let \(K\) be a knot with mosaic number \(m(K) > 2\), and let \(M_K\) be a planar mosaic representation of \(K\) on an \(m(K) \times m(K)\) grid. In a planar knot mosaic, all tiles with 4 connection points must be located within the interior \((m(K)-2) \times (m(K)-2)\) sub-mosaic to ensure suitable connectivity, so the remaining $4m(K)-4$ tiles on the boundary of our mosaic can be $T_0$ through $T_6$, but not any of $T_7$ through $T_{10}$.

Map the interior \((m(K)-2) \times (m(K)-2)\) sub-mosaic to any face of a spherical $(m(K)-2)$-mosaic as in \autoref{fig:m(K)-2}. There will be $4m(K)-4$ tiles hanging off the sides of our spherical $(m(K)-2)$-mosaic. We can remove all corner tiles and glue the remaining boundary tiles to the spherical $(m-2)$-mosaic by considering two cases:

\textit{Case 1:} If a corner tile of \(M_K\) is $T_0$, then we can remove it and preserve the knot type and closed-ness of our knot diagram by gluing the non-corner boundary tiles on the appropriate adjacent cube faces.

\textit{Case 2:} If a corner of \(M_K\) is one of $T_1$ through $T_4$, then we can still remove it and glue the non-corner boundary tiles on adjacent cube faces because doing so re-closes diagram as the two connection points previously touching the corner tile will now touch each other. 

Thus, omitting the corner tile and gluing the non-corner boundary tiles to the appropriate adjacent cube faces will still yield a single closed component preserving the knot type of our knot diagram in all cases. Visualizations of this casework can be found in \autoref{fig:m(K)-2cases}.

\vspace{1em}

This shows we have the following upper bound on the spherical mosaic number in terms of the mosaic number:
\[
sm(K) \leq m(K) - 2.
\]
\end{proof}

\begin{figure}[ht]
    \centering

\begin{minipage}{0.28\textwidth}
\centering
\includegraphics[width=0.6\linewidth]{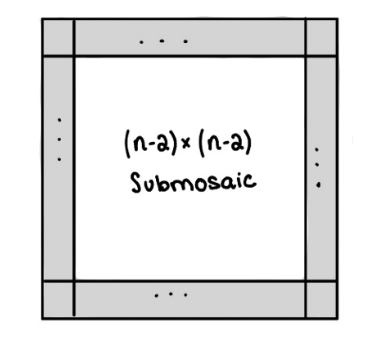}
\par\smallskip
{\footnotesize\textit{Take an $n$-mosaic and identify the $(n-2)\times(n-2)$ submosaic which must contain all crossing information.}}
\end{minipage}
\hspace{0.04\textwidth}
\begin{minipage}{0.28\textwidth}
\centering
\includegraphics[width=0.6\linewidth]{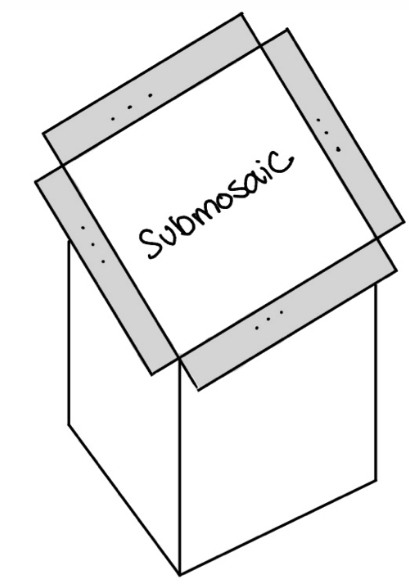}
\par\smallskip
{\footnotesize\textit{Remove the corner tiles and glue the submosaic to any face of a spherical $(n-2)$-mosaic.}}
\end{minipage}
\hspace{0.04\textwidth}
\begin{minipage}{0.28\textwidth}
\centering
\includegraphics[width=0.6\linewidth]{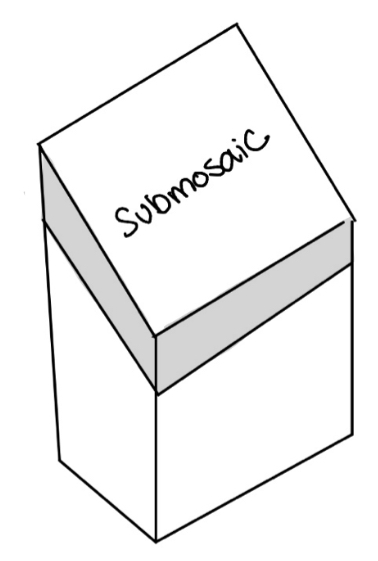}
\par\smallskip
{\footnotesize\textit{Fold and glue the remaining boundary tiles of the mosaic along adjacent faces. This preserves connectedness.}}
\end{minipage}
\caption{Visual depiction of gluing an $n$-mosaic onto a spherical mosaic.}
    \label{fig:m(K)-2}
\end{figure}

The key idea in this proof suggests that further refinements of this bound might be possible depending on the spatial distribution of the crossings and the structure of the boundary tiles. For example, it may be that a planar mosaic representation of that knot where all the $T_7$, $T_8$, $T_9$, and $T_{10}$ tiles are located in a $k \times k$ sub-mosaic has $sm(K) \le k$. However, this relies on the behavior of the paths between the connection points on the boundary of our sub-mosaic which while not having any crossings may complicate matters. This is a good direction for future research.

\begin{figure}[ht]
    \centering

\noindent
\begin{minipage}[t]{0.42\textwidth}
\textbf{Case 1}

\includegraphics[width=\linewidth]{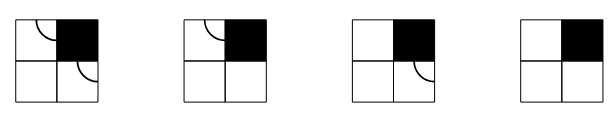}
\par\smallskip
{\footnotesize\textit{Consider the leftmost corner of a planar mosaic whose bottom left corner tile is the empty tile $T_0$. There are 4 possible such cases.}}

\vspace{0.8cm}

\includegraphics[width=\linewidth]{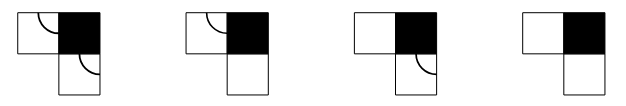}
\par\smallskip
{\footnotesize\textit{Delete the bottom left tile.}}

\vspace{0.8cm}

\includegraphics[width=\linewidth]{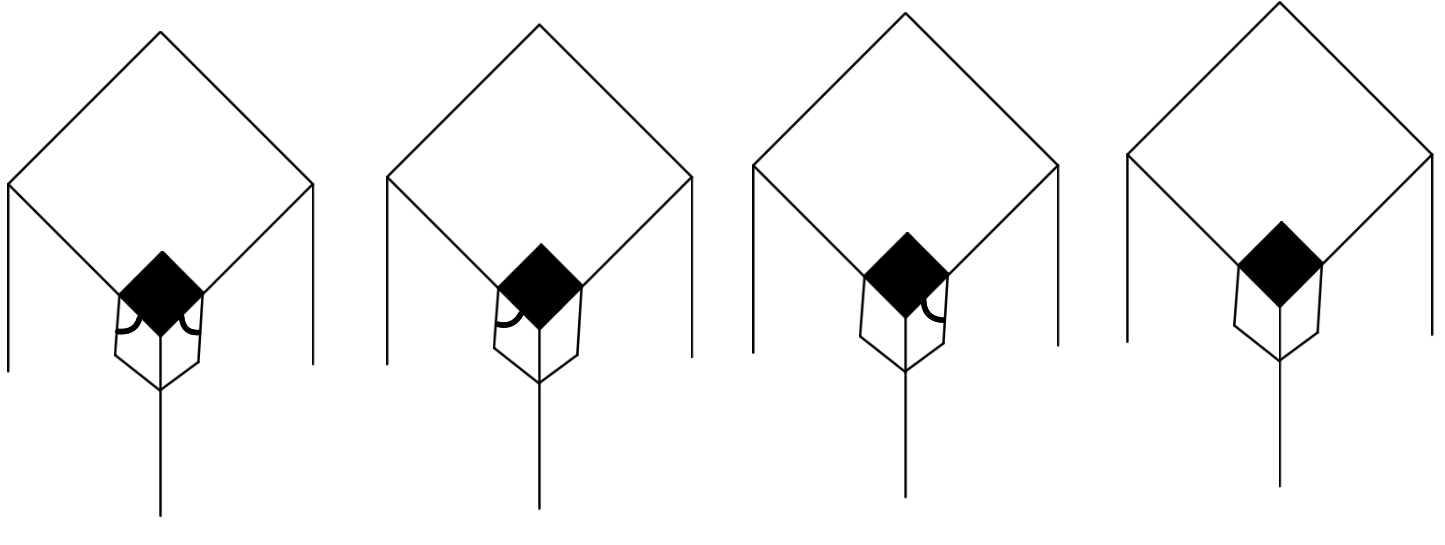}
\par\smallskip
\end{minipage}
\hfill
\begin{minipage}[t]{0.42\textwidth}
\textbf{Case 2}

\includegraphics[width=\linewidth]{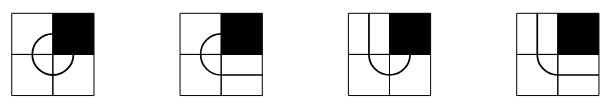}
\par\smallskip
{\footnotesize\textit{Consider the leftmost corner of a planar mosaic whose bottom left corner tile is $T_3$. There are 4 possible subcases.}}

\vspace{1.3cm}

\includegraphics[width=\linewidth]{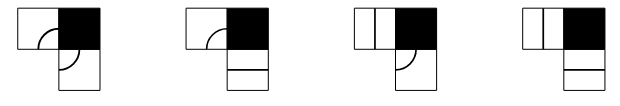}
\par\smallskip
{\footnotesize\textit{Delete the bottom left tile.}}

\vspace{0.8cm}

\includegraphics[width=\linewidth]{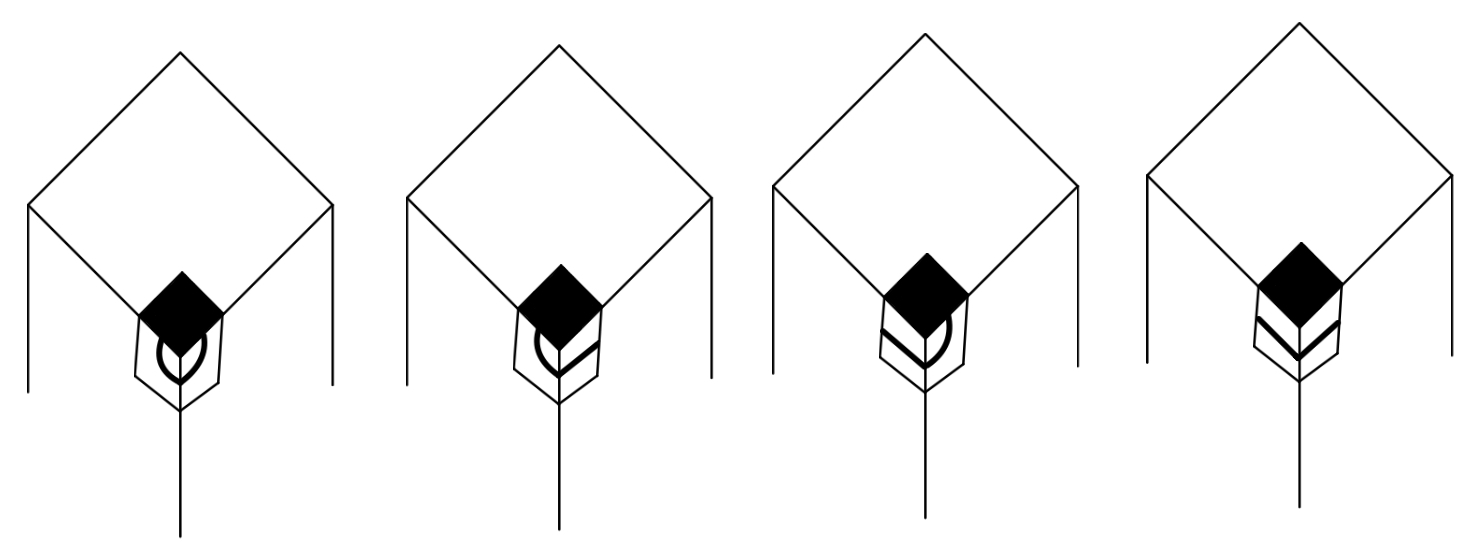}
\par\smallskip
\end{minipage}\\
{\footnotesize\textit{Glue the adjacent boundary tiles to the appropriate tiles. This preserves knot type.}}
\caption{Visual depiction of casework in \autoref{m(K))-2}. In all cases, the resultant $SM_K$ is a suitably connected mosaic that preserves the knot type of $K$. We show this for the bottom left corner, but the same logic holds for all 4 corners.}

    \label{fig:m(K)-2cases}
\end{figure}

\hspace{.5in}
\begin{theorem}\label{t(K)-1}
     Given a knot, $K$, $st(K) \le t(K)-1$.
\end{theorem}

\begin{proof}
    Consider an arbitrary knot $K$ and a corresponding planar space efficient $n$-mosaic. This must have a $T_3$ tile such that there are no non-empty tiles in the entries both below or to the left of the $T_3$ tile. Call this tile $T^*$. Such a representation will always exist but may not be a $m(K)-$mosaic. We can then represent $K$ on a spherical $(n-1)-$mosaic with $t(K)-1$ tiles by deleting $T^*$, gluing all tiles above and to the right of $T^*$ (denoted by $B$ in \autoref{fig:t(K)-1}) to the top face of the spherical $(n-1)-$mosaic such that the tile to upper right diagonal to $T^*$ is glued to the bottom left corner of the top face of the spherical $(n-1)-$mosaic, gluing all tiles in the same row or below and to the right of $T^*$ (denoted by $C$ in \autoref{fig:t(K)-1}) to the front right face of the spherical $(n-1)-$mosaic such that the tile to the right of $T^*$ is glued to the top left corner of the front right face of the spherical $(n-1)-$mosaic, and gluing all tiles above and in the same column or to the left $T^*$ tile (denoted by $A$ in \autoref{fig:t(K)-1}) to the front left face of the spherical $(n-1)-$mosaic such that the tile above the $T_3$ tile is glued to the top right corner of the front left face of the spherical $(n-1)-$mosaic. Thus we have represented $K$ as a spherical knot mosaic with $t(K)-1$ tiles. 
\end{proof}

\begin{figure}
    \centering

\begin{minipage}{0.28\textwidth}
\centering
\includegraphics[width=0.8\linewidth]{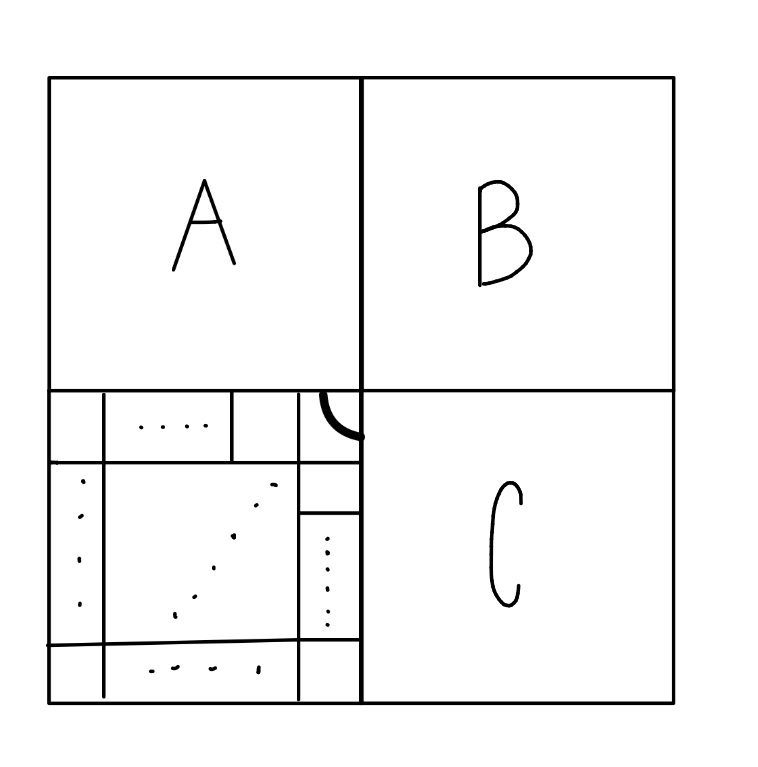}
\par\smallskip
{\footnotesize\textit{A general representation of an $n$-mosaic satisfying the conditions in \autoref{t(K)-1}.}}
\end{minipage}
\hspace{0.04\textwidth}
\begin{minipage}{0.28\textwidth}
\centering
\includegraphics[width=1.0\linewidth]{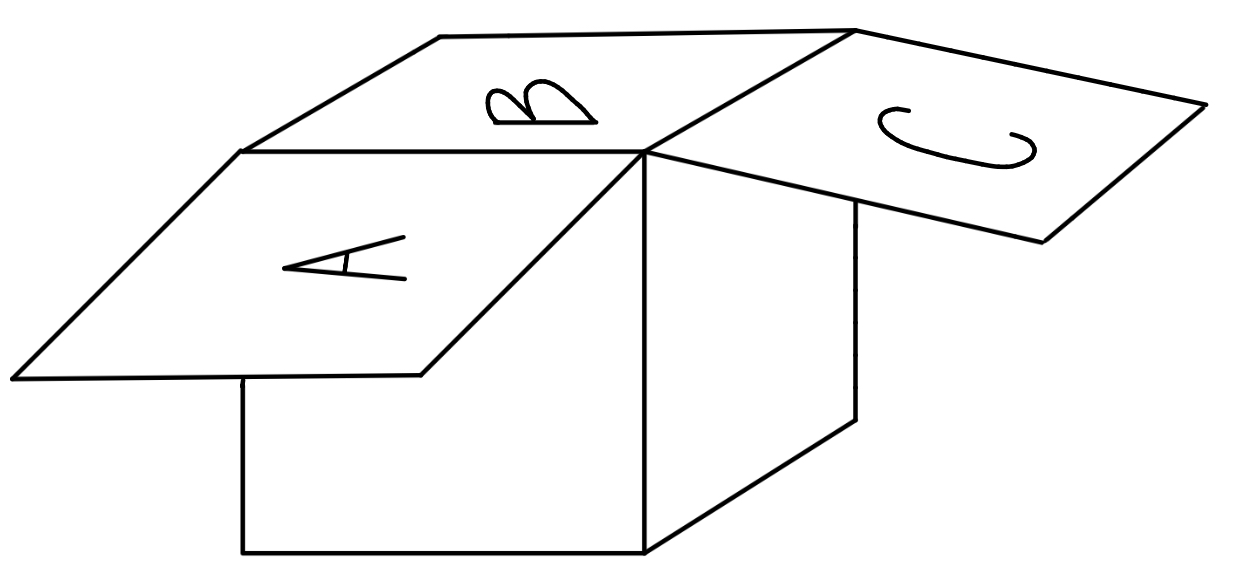}
\par\smallskip
{\footnotesize\textit{First glue $B$ to the top face of the cube, delete the bottom left portion of the mosaic which is all empty tiles except for $T^{*}$ in the top right.}}
\end{minipage}
\hspace{0.04\textwidth}
\begin{minipage}{0.28\textwidth}
\centering
\includegraphics[width=0.8\linewidth]{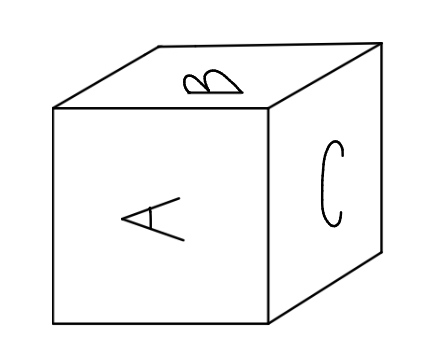}
\par\smallskip
{\footnotesize\textit{Glue $A$ to the front left face and $C$ to the front right face. This preserves knot type.}}
\end{minipage}

    \caption{Visual depiction of \autoref{t(K)-1}. Note that $A$, $B$, and $C$ do not actually need to be the same dimensions as each other, but each of $A$, $B$, and $C$ will be at most an $(n-1)$-mosaic.}
    \label{fig:t(K)-1}
\end{figure}

This bound is sharp for the unknot, but computational evidence has shown it is not for many nontrivial knots. In many cases a better bound is obvious from a space efficient knot mosaic, but for others, you cannot get a better bound directly from a space efficient knot mosaic and other methods must be employed. This is also a good direction for future research.

\begin{definition}
Let the boundary of a cube be tiled such that each face is an $n \times n$ mosaic grid, and all tiles conform to the standard mosaic tile set and connectivity rules. A \emph{mosaic belt} is a sequence of tiles satisfying the following conditions:
\begin{enumerate}
    \item Each tile in the sequence has exactly two connecting edges to adjacent tiles.
 \item Consecutive tiles in the sequence share an edge and connected through opposite edges so that the sequence forms a continuous, non-branching path with no turns or gaps.
    \item The complement of the mosaic belt is disconnected.
\end{enumerate}
\end{definition}

%come back for spacing

\begin{figure}[!ht]
    \centering
    \includegraphics[width=0.65\linewidth]{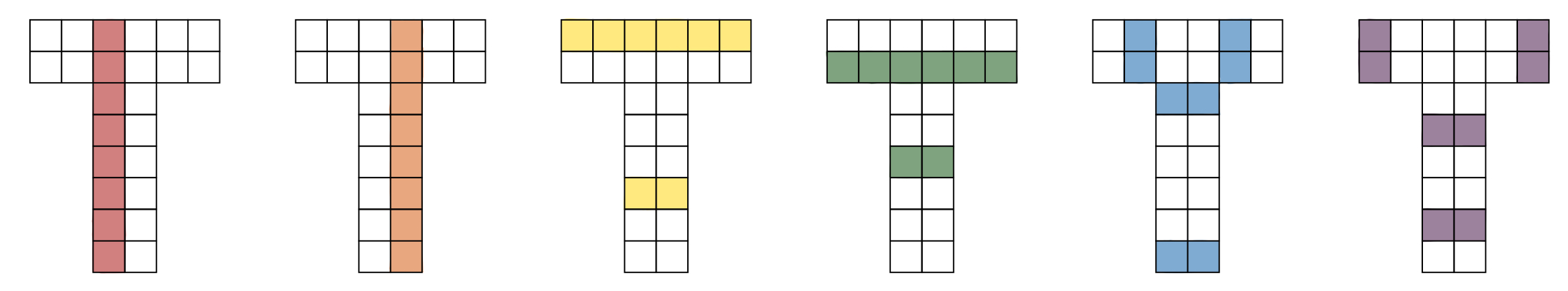}
    \caption{The six distinct mosaic belts of the empty spherical 2-mosaic}
    \label{fig:placeholder}
\end{figure}

\begin{theorem}\label{crossingnumbound}
    If $SM_K$ is a spherical knot $n$-mosaic for a knot, $K$, then the maximum number of crossing tiles in $SM_K$ is $ 6n^2-3n+1$.
\end{theorem}

\begin{proof}
We will start with a spherical $n$-mosaic containing only crossing tiles, so it is entirely tiled with $T_9$ or $T_{10}$. There are $6n^2$ total tiles. This is not a spherical knot $n$-mosaic, but instead is a spherical link $n$-mosaic with exactly $3n$ connected components as each connected component corresponds to one of the $3n$ mosaic belts. We will replace each crossing tile one at a time with $T_7$ or $T_8$ until there is only one connected component remaining.

Choose any mosaic belt. Now choose any tile on that belt and replace it with a non-crossing tile with four connection points, either $T_7$ or $T_8$. This change reduces the number of connected components in our spherical link $n$-mosaic by one and means the tile opposite the one we changed now represents a crossing of a single component instead of two disjoint components.

We can then iterate this process by choosing another crossing tile that is part of our larger connected component which does not represent the component crossing itself and changing it to a non-crossing tile with four connection points, this again reduces the number of connected components by one. We can do this until there is only one connected component and thus we have a spherical knot $n$-mosaic for some knot with $6n^2 - (3n-1)$ crossing tiles as there were originally $6n^2$ crossing tiles and we removed $3n-1$ of them. This is therefore the most crossings possible for a spherical knot $n$-mosaic.
\end{proof}

\begin{remark}
    \autoref{crossingnumbound} is an equivalent result to \cite[Theorem 3.2 p. 634]{CubicMosaics} 
\end{remark}

\begin{corollary}\label{cor:alternating}
    For each $n \in \mathbb{N}$ there is at least one knot, $K$, and representation of that knot as a spherical knot $n$-mosaic such that $K$ is alternating and $$6n^2-3n+1 \ge c(K).$$ 
\end{corollary}

\begin{proof}
    We modify the proof of \autoref{crossingnumbound} such that the original tiling with $6n^2$ crossing tiles alternates. That is, for each component every under crossing is preceded by and followed by an over crossing and vice versa. Now notice that replacing any crossing tile involving two components with $T_7$ or $T_8$ preserves the alternating nature of the diagram, thus after replacing $3n-1$ crossing tiles we are left with an alternating knot diagram with $6n^2-3n+1$ crossings. Note this may not be a reduced knot diagram so we do not necessarily find that this is the crossing number, but instead an upper bound for the crossing number of our knot.
\end{proof}

\begin{remark}
    One can find a stronger result in \cite[Theorem 3.3 p. 636]{CubicMosaics}. They show there will always exist a knot such that $sm(K)=n$ and $c(K)=6n^2-3n+1$.
\end{remark}

\begin{figure}[!ht]
    \centering
   \includegraphics[scale=0.8]{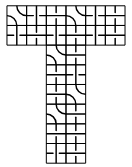}
\caption{An example of a spherical knot $2$-mosaic which has the maximum possible number of crossing tiles, $6\cdot 2^2-3\cdot 2+1 = 19$}
    \label{fig:maximal19}
\end{figure}

\begin{corollary}\label{lower-bound}
If $K$ is a knot with spherical mosaic number $sm(K)$, then $6(sm(K))^2-3(sm(K))+1 \ge c(K)$.
\end{corollary}

From \autoref{lower-bound} we have the following two results.

\begin{corollary}
    Let $SM_K$ be a spherical knot $1$-mosaic for some knot $K$, then $4 \ge c(K)$.
\end{corollary}

\begin{corollary}\label{cor:ge2}
    Let $K$ be a knot such that $c(K) \ge 5$, then $sm(K) \ge 2$.
\end{corollary}

We summarize these and further results of \autoref{lower-bound} in \autoref{tab:crossingboundresults}.
    \begin{table}[htbp]
    \centering
    
    \begin{tabular}{|c|c|}
        \hline
        If $sm(K)$, & then $c(K)$\\
        \hline
        \hline
        1 & $\le 4$\\
        \hline
        2 & $\le 19$\\
        \hline
        3 & $\le 46$\\
        \hline
        4 & $\le 85$\\
        \hline
        5 & $\le 136$\\
        \hline
        6 & $\le 199$\\
        \hline
        7 & $\le 274$\\
        \hline
        8 & $\le 361$\\
        \hline
        9 & $\le 460$\\
        \hline
        10 & $\le 571$\\
        \hline
        11 & $\le 694$\\
        \hline
        12 & $\le 829$\\
        \hline
        13 & $\le 976$\\
        \hline
        14 & $\le 1135$\\
        \hline
        15 & $\le 1306$\\
        \hline
    \end{tabular}
    \hspace{.5in}
    \begin{tabular}{|c|c|}
        \hline
        If $c(K)$, & then $sm(K)$\\
        \hline
        \hline
        $\ge 5$ & $\ge 2$\\
        \hline
        $\ge 20$ & $\ge 3$\\
        \hline
        $\ge 47$ & $\ge 4$\\
        \hline
        $\ge 86$ & $\ge 5$\\
        \hline
        $\ge 137$ & $\ge 6$\\
        \hline
        $\ge 200$ & $\ge 7$\\
        \hline
        $\ge 275$ & $\ge 8$\\
        \hline
        $\ge 362$ & $\ge 9$\\
        \hline
        $\ge 461$ & $\ge 10$\\
        \hline
        $\ge 572$ & $\ge 11$\\
        \hline
        $\ge 695$ & $\ge 12$\\
        \hline
        $\ge 830$ & $\ge 13$\\
        \hline
        $\ge 977$ & $\ge 14$\\
        \hline
        $\ge 1136$ & $\ge 15$\\
        \hline
        $\ge 1307$ & $\ge 16$\\
        \hline
    \end{tabular}
    \caption{Several consequences of \autoref{lower-bound}.}
    \label{tab:crossingboundresults}
\end{table}

There are several more questions that spherical knot mosaics give rise to which the authors plan to explore in subsequent papers.

\section*{Acknowledgments}
The authors would like to thank Jennifer Schultens for her encouragement and advice throughout this project. We also thank Aaron Heap for his feedback and keen eye in catching some errors in the appendix of a previous version of this paper. The second author would like to thank Samantha Pezzimenti for introducing him to knot mosaics during the Unknot V conference at Seattle University and inspiring their work in this direction. Both authors thank Samantha Pezzimenti for her comments on and benignity of a previous version of this paper.

\bibliographystyle{plain}
\bibliography{references}

\appendix
\section{Planar Representation Examples}\label{app:planar}
Planar representations of the first 35 unoriented prime knots as spherical knot mosaics. Due to \autoref{lower-bound}, each of these planar representations realize their knot type's spherical mosaic number, but in general these may not be minimally space efficient.

\begin{center}
    \includegraphics[scale=1]{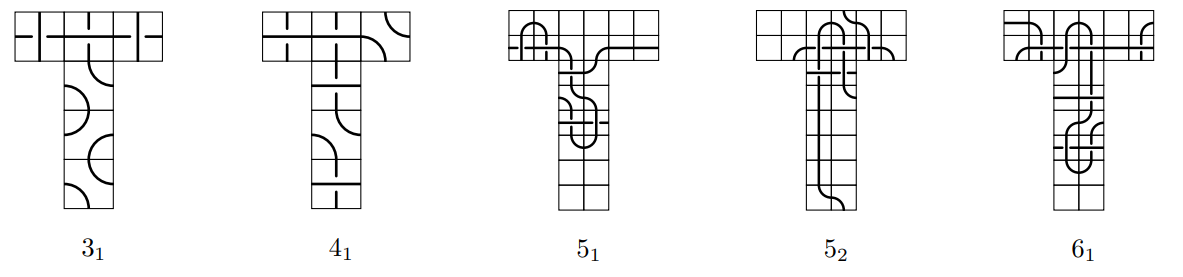}\\
    \includegraphics[scale=1]{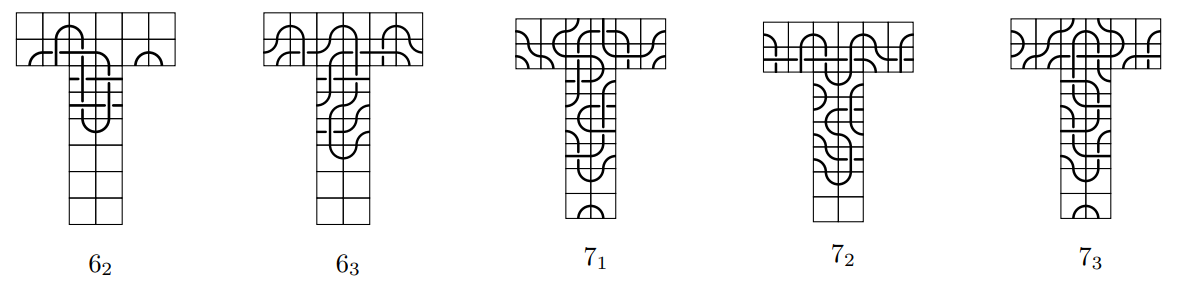}\\
    \includegraphics[scale=1]{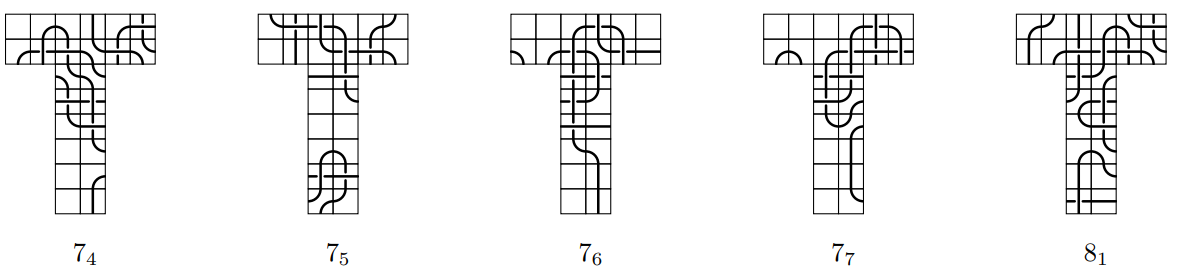}\\
    \includegraphics[scale=1]{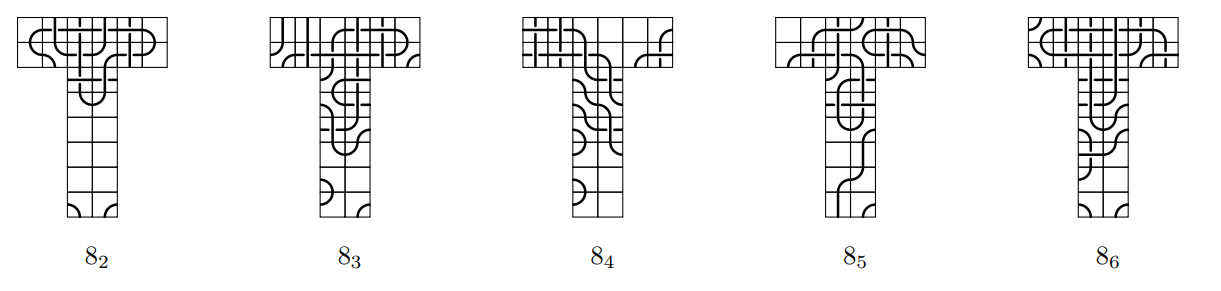}\\
    \includegraphics[scale=1]{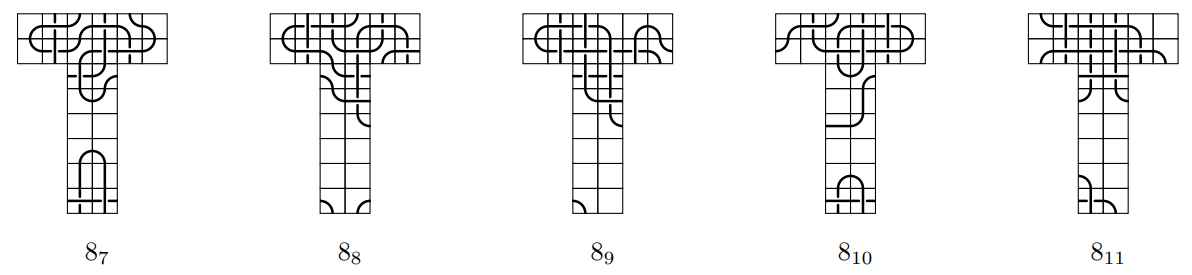}\\
    \includegraphics[scale=1]{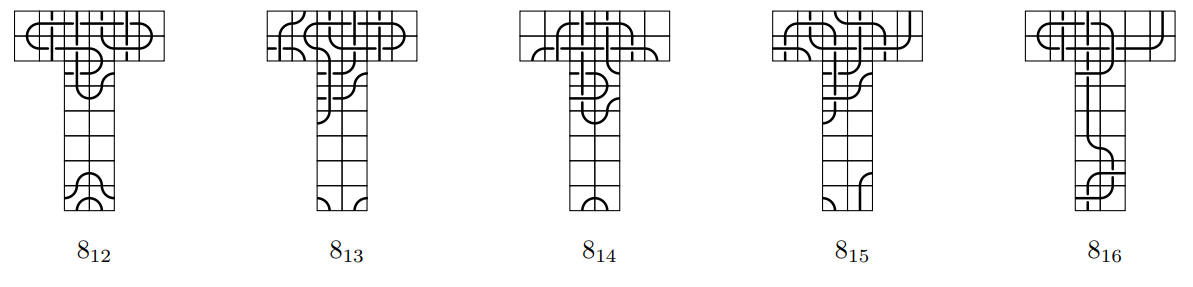}\\
    \includegraphics[scale=1]{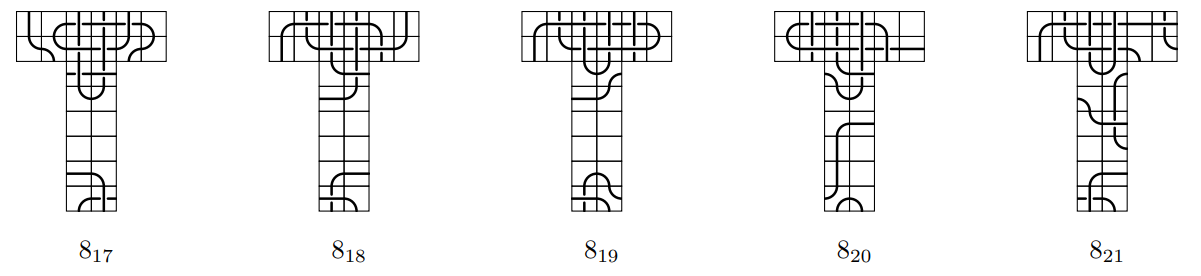}\\
\end{center}

\end{document}